\documentclass[10pt]{extarticle}
\usepackage[english]{babel}
\usepackage{amsmath,amssymb,enumerate,bbm,latexsym,alltt,theorem,mathabx,listings}
\usepackage{authblk}
\newcommand{\assign}{:=}
\newcommand{\mathd}{\mathrm{d}}
\newcommand{\nobracket}{}
\newcommand{\tmmathbf}[1]{\ensuremath{\boldsymbol{#1}}}
\newcommand{\tmop}[1]{\ensuremath{\operatorname{#1}}}
\newcommand{\tmtextit}[1]{{\itshape{#1}}}

\newenvironment{proof}{\noindent\textbf{Proof\ }}{\hspace*{\fill}$\Box$\medskip}
\newenvironment{proof*}[1]{\noindent\textbf{#1\ }}{\hspace*{\fill}$\Box$\medskip}

\newenvironment{tmparmod}[3]{\begin{list}{}{\setlength{\topsep}{0pt}\setlength{\leftmargin}{#1}\setlength{\rightmargin}{#2}\setlength{\parindent}{#3}\setlength{\listparindent}{\parindent}\setlength{\itemindent}{\parindent}\setlength{\parsep}{\parskip}} \item[]}{\end{list}}

\newtheorem{corollary}{Corollary}
\newtheorem{definition}{Definition}
\newtheorem{lemma}{Lemma}
\newtheorem{proposition}{Proposition}
{\theorembodyfont{\ttfamily}\newtheorem{algorithm}{Algorithm}}
{\theorembodyfont{\rmfamily}\newtheorem{remark}{Remark}}
\newtheorem{theorem}{Theorem}
\numberwithin{equation}{section}
\numberwithin{lemma}{section}
\numberwithin{remark}{section}
\numberwithin{algorithm}{section}
\begin{document}

\title{Asymptotic expansions of the inverse of the Beta distribution}
\author{Dimitris Askitis\footnote{email: dimitrios@math.ku.dk}}
\affil{ Department of Mathematical Sciences\\ University of Copenhagen}

\maketitle

\begin{abstract}
In this work in progress, we study the asymptotic behaviour of the $p$-quantile of
  the Beta distribution, i.e. the quantity $q$ defined implicitly by
  $\int_0^q t^{a - 1} (1 - t)^{b - 1} \text{d} t = p B (a, b)$, as a function of
  the first parameter $a$. In particular, we derive
  asymptotic expansions of and $q$ and its logarithm at $0$ and $\infty$. Moreover, we provide some relations between Bell and N{\o}rlund Polynomials, a
  generalisation of Bernoulli numbers. Finally, we provide Maple and Sage
  algorithms for computing the terms of the asymptotic expansions.
\end{abstract}



\section{Introduction}

\subsection{Background}
Granted a probability distribution on $\mathbbm{R}$, its median is defined as
the value $m \in \mathbbm{R}$ that leaves exactly half of the ``mass'' of the
distribution on its left and half on its right. Instead of requiring that $m$
splits the mass exactly in two equal parts, one may choose a $p \in [0, 1]$
and define the more general notion of the $p$-quantile value of the probability
distribution:\cite{logconc}
\begin{definition}
  Let $F$ be a cumulative distribution function on some subset $I \subset
  \mathbb{R}$. Let $p \in [0, 1]$. A $p$-quantile of $F$ is a point $q \in I$
  such that $F (q) = p$. If $p = 1 / 2$, a $1 / 2$-quantile is called
  median.
\end{definition}
For an arbitrary probability distribution on $\mathbbm{R}$, not always do
$p$-quantiles exist, neither do they have to be unique, but for a
distribution with density wrt to Lebesgue measure $p$-quantile values always
exist, as then the distribution function is continuous and increasing, and if
furthermore the density is a.e. non-zero, they are also unique, as the
distribution function shall be strictly increasing.

One point of interest has been the study of the $p$-quantiles, including
medians, of a parametrised family of probability distributions as a function
of the parameter, given a fixed value of $p$. Such a function is well defined
if the distribution has density wrt to the Lebesgue measure which is a.e.
non-zero. Questions that may arise in this context have to do with
analyticity, monotonicity, geometric properties and approximations, in
particular asymptotic expansions, of the implicit function $q (a)$ defined by
an equation of the form $F_a (q (a)) = p$, where $F_a$ is a family of
commulative distribution functions. Because of the implicit definition, the
study of its properties can be challenging. An example is the median of the
gamma distribution, which has been studied in several occasions, for example
in {\cite{chenrubin0}}, {\cite{chenrubin}}, and many connections have been
found, for example with the Ramanujan's rational approximation of $e^x$, see
{\cite{gammaram1}}, {\cite{gammaram2}} and {\cite{gammaram3}}, while in
{\cite{mediangammaconvex}} it was also proved that it is is a convex function.

In this paper, considering $p$ fixed in $(0, 1)$, we focus on studying the
$p$-quantile of the beta distribution, i.e. the distribution on $[0, 1]$ with
the density function $t \mapsto t^{a - 1} (1 - t)^{b - 1}$, as a function of
the parameter $a$ considering $b$ fixed. The $p$-quantile of the beta distribution has been considered
by Temme in {\cite{temme}}, who studied the asymptotic behaviour of the
$p$-quantile (or in his notation, the inverse of the normalised beta incomplete
function) under restrictions over relations between the two parameters of the
beta distribution. Also, see {\cite{beta1}} for some inequalities on the
median. This preprint is to be a continuation of our work in \cite{logconc}, which deals with convexity/concavity properties.

The $p$-quantile of the beta distribution, as a function of the first
parameter, is defined as:
\begin{definition}
  Fix $p \in (0, 1)$ and $b \in (0, + \infty)$. The function $q : (0, +
  \infty) \rightarrow (0, 1)$ defined implicitly by
  \begin{equation}
    \label{betamaineq} \int_0^{q (a)} t^{a - 1}  (1 - t)^{b - 1} dt = p
    \int_0^1 t^{a - 1}  (1 - t)^{b - 1} dt
  \end{equation}
  is called the $p$-quantile of the beta distribution with parameters $a$ and
  $b$.
\end{definition}
As in {\cite{chenrubin}} for the case of the median of the gamma distribution,
to study the $p$-quantile we consider and study an auxilliary function related
to its logarithm
\begin{equation}
  \varphi (a) \assign - a \log q (a)
\end{equation}
and it will become clear that studying the logarithm gives more information on
the behaviour of the $p$-quantile. One may also consider $\varphi$ itself as the
$(1-p)$-quantile of some distribution. Indeed, using change of variables in
\eqref{betamaineq}
\begin{equation}
  \int_0^{\varphi (a)} e^{- s} (1 - e^{- s / a} )^{b - 1} \mathd s = (1-p)
  \int_0^{\infty} e^{- s} (1 - e^{- s / a})^{b - 1} \mathd s
\end{equation}
Later, Bernoulli numbers and a generalisation of them known as N{\o}rlund
polynomials will become useful. The Bernoulli numbers $B_n$ are classically
defined through their generating function
\begin{equation}
  \frac{x}{e^x - 1} = \sum_{n = 0}^{\infty} B_n \frac{x^n}{n!}
\end{equation}
They can be generalised to the Bernoulli polynomials $B_n (t)$, defined
similarily through the generating function
\begin{equation}
  \frac{x e^{t x}}{e^x - 1} = \sum_{n = 0}^{\infty} B_n (t) \frac{x^n}{n!}
\end{equation}
Another generalisation of Bernoulli numbers are the N{\o}rlund polynomials
$B_n^{(s)}$ defined through the generating function
\begin{equation}
  \left( \frac{x}{e^x - 1} \right)^s = \sum_{n = 0}^{\infty} B^{(s)}_n
  \frac{x^n}{n!} \label{norlund}
\end{equation}
They are polynomials in $s$. If $s \in \mathbbm{N}$, then $B_n^{(s)}$ is the
$s$-fold convolution of Bernoulli numbers. An account on N{\o}rlund
polynomials can be found in {\cite[24.16]{DLMF}} and references within.
Bernoulli and N{\o}rlund appear often when we consider asymptotic expansions
of the gamma and related functions (see e.g. {\cite{tricomi}}.
\subsection{Main results}
We state the following propositions regarding first order asymptotics. They are proved in \cite{logconc}. In the rest, $\gamma_b$ denotes the $(1-p)$-quantile of the gamma distribution
with parameter $b$.
\begin{proposition}\cite[Proposition 1.2]{logconc}
  \label{prop1}The $p$-quantile of the beta distribution $q (a)$ is a real
  analytic, increasing function of $a$. It has limits
  \begin{align*} \lim_{a \rightarrow 0} q (a) = 0 \end{align*}
  and
  \begin{align*} \lim_{a \rightarrow \infty} q (a) = 1 \end{align*}
\end{proposition}
\begin{proposition}\cite[Proposition 1.3]{logconc}
  \label{prop2}The function $\varphi (a) = - a \log q (a)$ is real analytic
  and increasing for $b > 1$, constant for $b = 1$ and decreasing for $b < 1$.
  It has limits
  \begin{equation}
    \lim_{a \rightarrow 0} \varphi (a) = \log p \label{asymptotic0}
  \end{equation}
  and
  \begin{equation}
    \lim_{a \rightarrow \infty} \varphi (a) = \gamma_b
  \end{equation}
\end{proposition}
To study the asymptotic behaviour of $q$ and $\varphi$ in more depth, we shall
try to find the asymptotic expansions of $\varphi$ at $0$ and $\infty$.
Studying asymptotic expansions of implicit functions can be highly
non-trivial, as the method and the obstacles arising depend much on the form
of the defining implicit relation. For the $p$-quantile of the beta
distribution, we consider the cases of asymptotic expansions of $\varphi$
centered at $0$ and at $\infty$. In both cases, we shall combine
differentiation and Fa{\`a} di Bruno's formula \eqref{formfaa}, and the
existence of the expansion has to be proved inductively.

For the case of $0$, we shall compute the limits of the derivatives. For the
case of $\infty$, for the same purpose, we shall introduce the differential
operator $D$ defined by
\begin{equation*}
  D f (x) = x^2 \partial f (x) 
\end{equation*}
where $\partial$ denotes the common differentiation operator. The calculus of
$D$ is studied in subsection \ref{sectionD}.

This operator has the importance that it can give, under certain conditions,
the asymptotic expansion of a suitably smooth function at infinity, which is
summarized in the following lemma, which is proved in subsection
\ref{sectionD}:
\begin{lemma}
  \label{lemD}Let  $f \in C^n (0, \infty)$ for some $n \in \mathbbm{N}$.
  Then, the following hold:
  
  i. If $\lim_{x \rightarrow \infty} D^m f (x)$ exists in $\mathbbm{R}$ for
  all $m \leqslant n$, we have the asymptotic expansion
  \begin{align*} f (x) \sim \sum_{k = 0}^{n - 1} \frac{c_k}{x^k} +\mathcal{O} \left(
     \frac{1}{x^n} \right) \end{align*}
  where
  \begin{align*} c_k = \frac{(- 1)^k}{k!} \lim_{a \rightarrow \infty} D^k f (a), \quad m <
     n \end{align*}
  
  ii. Assume, conversely, that $f$ has asymptotic expansion of order $n$, i.e.
  $f (x) \sim \sum_{k = 0}^n \frac{c_k}{x^k} +\mathcal{O} \left(
  \frac{1}{x^{n + 1}} \right),$as well as that its derivatives $f^{(m)}$ admit
  asymptotic expansions of orders m+$n$, for $m \leqslant n$. Then, we have
  \begin{align*} c_k = \frac{(- 1)^k}{k!} \lim_{a \rightarrow \infty} D^k f (a) \end{align*}
\end{lemma}
We note that, if conditions in \tmtextit{i.} hold, we may apply the lemma to
$D^k f$ and get asymptotic expansions of higher derivatives, hence the
expansion in \tmtextit{i.} can be differentiated. Also, if in the previous
lemma $f \in C^{\infty} (0, \infty)$ and its conditions hold for all $n$, then
we may get the whole asymptotic expansion of $f$.

Regarding the functions $\varphi$ and $q$, we have the following two pairs of
theorems and Corollaries on their asymptotic expansions, which are proved in
sections 2 and 3 respectively. In the following, $\Psi (n, z) \assign
\partial^{n + 1} \tmop{Log} \Gamma (z)$ denotes the polygamma function. Also,
$(m)_n$ denotes the Pochhammer symbol of $m$, i.e. $(m)_n = m (m + 1) \ldots
(m + n - 1)$. If $m \in \mathbbm{N}$, then we have $(m)_n = \frac{(m + n - 1)
!}{(m - 1) !}$, and $(- m)_n = (- 1)^n \frac{m!}{(m - n) !}$ if $n \leqslant
m$, and $(- m)_n = 0$ if $n > m$. These identities will be widely used in this paper.
\begin{theorem}
  \label{theorem0}The function $\varphi$ admits the asymptotic expansion
  $\varphi (a) \sim \sum_{n = 0}^{\infty} c_n a^n$ at $0$, with $c_0 = \log p$
  and
  \begin{align*} c_{n } = \frac{\Psi (n-1, b) - \Psi (n-1, 1)}{n!} = \frac{(-
     1)^{n+1}}{n !} \int_0^{\infty} u^{n-1} \left( \frac{e^{- u} - e^{- bu}}{1 -
     e^{- u}} \right) \mathd u, \quad n \geq 1 \end{align*}
  For $b \in \mathbb{N}$, we have in particular
  \begin{align} \label{bintegerformula} c_{n } = (- 1)^{n+1} (n-1)! \sum_{k = 1}^{b - 1}
     \frac{1}{k^{n }} \end{align}
\end{theorem}
\begin{corollary}
  \label{cor0}An approximation for $\varphi$ for values of $a$ close to $0$ is
  \begin{align*} \varphi (a) \sim \log \frac{\Gamma (a + b)}{\Gamma (a + 1) \Gamma (b)} -
     \log p \end{align*}
  and for $q$
  \begin{align*} \frac{q (a)}{p^{1 / a}} \sim \left( \frac{\Gamma (a + b)}{\Gamma (a + 1)
     \Gamma (b)} \right)^{1 / a} \end{align*}
  each having a remainder term vanishing faster than $a^n$ at $0$, $\forall n
  \in \mathbbm{N}$. Hence, we have the asymptotic expansion
  \begin{equation}
    \frac{q (a)}{p^{1 / a}} \sim e^{- \gamma - \Psi (0, b)} \left( \sum_{n =
    0}^{\infty} \frac{\mathcal{B}_n (c_1, c_2, \ldots, c_n)}{n!} a^n \right)
  \end{equation}
  where $\gamma$ is the Euler constant, $c_n = \frac{\Psi (n - 1, b) - \Psi (n
  - 1, 1)}{n!}$ and $\mathcal{B}_n$ denotes the nth complete Bell polynomial
  (see Remark \ref{remarkbell}).
\end{corollary}
\begin{theorem}
  \label{theoreminf}The function $\varphi$ admits the asymptotic expansion
  $$\varphi (a) \sim \sum_{n = 0}^{\infty} \varphi_n \frac{(- 1)^n}{n!a^n},\quad a\rightarrow \infty$$ at
  $\infty$, with $\varphi_n$ satisfying the system of recursive relations
  \begin{align}
    \varphi_n =& - \sum_{j = 1}^{n - 1} \binom{n - 1}{j} \varphi_{n -
    j} \delta (0, j, 0) - \sum_{k = 0}^{n - 2} \sum_{j = 0}^k \binom{k}{j}
    \varphi_{k - j + 1} \delta (0, j, n - k - 1)\nonumber \\&+ B_n^{(1 - b)} \sum_{k =
    0}^{n - 1} (b + n - k)_k \gamma_b^{n - k} \label{recurs1}
  \end{align}
  \begin{equation}
    \delta (k, m, n) = \delta (k, m - 1, n + 1) + \sum_{j = 0}^{m - 1}
    \binom{m - 1}{j} \varphi_{m - j} \delta (k + 1, j, n) \label{recurs2}
  \end{equation}
  and the initial conditions
  \begin{equation}
    \varphi_0 = \gamma_b
  \end{equation}
  \begin{equation}
    \delta (k, 0, n) = B_n^{(1 - b)} \sum_{j = 0}^k \binom{k}{j} (- 1)^{k - j}
    (b + n - j)_j \gamma_b^{n - j}
  \end{equation}  
\end{theorem}
The recursive relations \eqref{recurs1} and \eqref{recurs2} in the foregoing
lemma work inductively. We know $\varphi_0$ and once we have computed
$\varphi_0, \ldots, \varphi_{n - 1}$, in order to compute $\varphi_n$ we use
\eqref{recurs1}, where the maximum of the second argument of $\delta$ that is
at most $n - 1$, and we can compute these terms using \eqref{recurs2} and the
initial conditions, as $\varphi_k$ appears there in orders at most equal to
the second argument of $\delta$, and that we already have computed. This
algorithm can give us the first terms of the asymptotic expansion:
\begin{align}
  \varphi (a) =& \gamma_b - \frac{\gamma_b (b - 1)}{2 a} + \frac{\gamma_b  (- 1
  + b)  \left( 7 \hspace{0.17em} b + \gamma_b - 5 \right)}{24 a^2}
  \hspace{0.17em}\nonumber \\&- \frac{\gamma_b  (- 1 + b)^2  \left( 3 \hspace{0.17em} b +
  \gamma_b - 1 \right)}{16 a^3} +\mathcal{O} \left( \frac{1}{a^4} \right)
\end{align}
Also, for $q$ we then get:
\begin{corollary}
  \label{corinf}For $a \rightarrow \infty$, an asymptotic expansion for $q$ is
  \begin{equation}
    q (a) \sim \sum_{n = 0}^{\infty} \frac{\mathcal{B}_n (-\varphi_0,
    2\varphi_1, -3\varphi_2, \ldots, (-1)^nn\varphi_{n - 1})}{n!}
    \frac{1}{a^n}
  \end{equation}
  where $\varphi_n$ is the sequence defined in Theorem \ref{theoreminf}.
\end{corollary}
In section 4 we state some relations
between N{\o}rlund, Bernoulli and Bell polynomials that we came upon and we
could not find in the literature. These relations come out by considering the
coefficients of Bernoulli generating functions as taylor coefficients, i.e. as
limits of derivatives, and using Fa{\`a} di Bruno's formula, and its relation
to Bell polynomials, to compute these derivatives. Finally, in the appendix we
implement the recursive relations of Theorem \ref{theoreminf} as Maple and
Sage algorithms and give coefficients of asymptotic expansions for some
specific values.
\section{Asymptotics at $0$}

For computing the asymptotic expansion of $\varphi$ at $0$, our method
consists of iterated differentiation of relations that implicitely contain the
$p$-quantile and use Fa{\`a} di Bruno's formula. Then, taking limits for $a
\rightarrow 0$ and computing the limits of all the terms, we compute the
limits of the derivatives which then wields the asymptotic expansion, as, if
$f \in C^{\infty} (0, \varepsilon)$, for some $\varepsilon > 0$, and $\lim_{x
\rightarrow 0} f^{(n)} (x)$ exists in $\mathbbm{R}$ for all $n$, denoting this
limit by $f^{(n)} (0)$ we have $f \sim \sum_{k = 0}^{\infty} \frac{f^{(n)}
(0)}{n!} x^n$. The converse is not necessarily valid: if $f$ admits asymptotic
expansion at $0$ it is not necessary that the limits of the derivatives exist,
as there may be oscillations. The limits of the derivatives of $\varphi$ will
be computed then inductively.

First, we use integration by parts in \eqref{betamaineq} getting
\begin{equation}
  e^{- \varphi (a)}  (1 - q (a))^{b - 1} + (b - 1)  \int_0^{q (a)} t^a  (1 -
  t)^{b - 2} dt = ap \frac{\Gamma (b) \Gamma (a)}{\Gamma (a + b)} = p \Gamma
  (b) W (a) \label{eq00}
\end{equation}
where $W (a) = \Gamma (a + 1) / \Gamma (a + b)$. This function $W$ is studied
in {\cite[C4]{special_functions}}, and in a generalised form in
{\cite{ratiogamma}}, where several properties, such as complete monotonicity,
are proved. We consider the logarithmic derivative of $W$ and we note that, as
$\log W (a) = \log \Gamma (a + 1) - \log \Gamma (a + b)$, by the integral
representation of the digamma function $\Psi (0, z)$ (see {\cite[Theorem
1.6.1]{special_functions}}), we get that
\begin{equation}
  \label{Wformula} (\log W (a))' = \Psi (0, a + b) - \Psi (0, a + 1) = -
  \int_0^{\infty} e^{- au} \left( \frac{e^{- u} - e^{- bu}}{1 - e^{- u}}
  \right) \mathd u
\end{equation}
We define the function $\psi$ by
\begin{equation}
  \label{psidef} \psi (a) \assign - \log \frac{\Gamma (a + 1)}{\Gamma (a + b)}
  - \log p \Gamma (b) = - \log W (a) - \log p \Gamma (b)
\end{equation}
which implies that $e^{- \psi (a)} = p \Gamma (b) W (a)$ and \eqref{eq00} can
be rewritten as
\begin{equation}
  \label{eq0b} e^{- \varphi (a)}  (1 - q (a))^{b - 1} + (b - 1)  \int_0^{q
  (a)} t^a  (1 - t)^{b - 2} \mathd t = e^{- \psi (a)}
\end{equation}
Hence, as $\psi \in C^{\infty} (- 1, \infty)$, denoting the limit of the $k$th
derivative of $\psi$ at $0$ by $\psi^{(k)} (0)$, by \eqref{Wformula} we have
\begin{align}
  \psi^{(k)} (0) = \Psi (k - 1, b) - \Psi (k - 1, 1) = (- 1)^{k - 1}
  \int_0^{\infty} u^{k - 1} \left( \frac{e^{- u} - e^{- bu}}{1 - e^{- u}}
  \right) \mathd u
\end{align}
Let denote by $\varphi^{(n)} (0)$ the right limit of $\varphi^{(n)}$ at $0$,
supposing it exists. We already have, combining \eqref{asymptotic0} and
\eqref{psidef}, that $\varphi (0) = \psi (0) = - \log p$. Our goal is to prove
that for the limits of all the derivatives of $\varphi$ and $\psi$ at $0$ are
the same, i.e. we have $\varphi^{(k)} (0) = \psi^{(k)} (0)$.
Differentiating \eqref{eq0b} we get
\begin{align*}
  - \psi' (a) e^{- \psi (a)} + \varphi' (a) e^{- \varphi (a)}  (1 - q (a))^{b
  - 1} = (b - 1)  \int_0^{q (a)} t^a  (1 - t)^{b - 2} \log t \mathd t
\end{align*}
We define the functions
\begin{equation}
  \rho (a) \assign (1 - q (a))^{b - 1}
\end{equation}
\begin{equation}
  \sigma (a) \assign \int_0^{q (a)} t^a  (1 - t)^{b - 2} \log t \mathd t
\end{equation}
and hence the last equation can be rewritten as
\begin{equation}
   \label{eq1} - \psi' (a) e^{- \psi (a)} + \varphi' (a) e^{- \varphi
  (a)} \rho (a) = (b - 1) \sigma (a)
\end{equation}
We will use this equality to find the limits of the derivatives of $\varphi$.
This will be done inductively, differentiating \eqref{eq1} at each step. Our
strategy is, at the $k$th step, where we will want to compute the limit of the
$k + 1$ derivative, that we use the results from the previous steps about the
asymptotic behaviour of $\varphi$ up to the $k$th derivative to find the
asymptotic behaviour of the derivatives of $q$ up to $k$, and then use this
result to find the behaviour of the derivatives of $g$ and $h$ up to $k$, so
that we finally compute the limit of the $k + 1$ derivative of $\varphi$. The
first part will be done in the next lemmas, and the inductive proof will be
given in the end of the section.

We state the following well known differentiation formulas that we will be
constantly using, see (1.4.12) and (1.4.13) in \cite{DLMF}: The product formula for derivation,
\begin{equation}
  \label{formprod} \left( \prod_{i = 1}^k f_i (x) \right)^{(n)} = \sum_{\{
  \tmmathbf{j} \in \mathbb{N}^k \mid \sum_{i = 1}^k j_i = n\}} \binom{n}{j_1,
  j_2, ..., j_k} \prod_{i = 1}^k f_i^{(j_i)} (x)
\end{equation}
and the Fa{\`a} di Bruno formula, for the derivatives of composite functions,
\begin{equation}
  \label{formfaa} (f \circ g)^{(n)} (x) = \sum_{\{ \tmmathbf{m} \in
  \mathbb{N}^n \mid \sum_{j = 1}^n jm_j = n\}} \frac{n!}{m_1 !m_2 !...m_n !}
  f^{(\sum_{j = 1}^n m_j)} (g (x))  \prod_{j = 1}^n \left( \frac{g^{(j)}
  (x)}{j!} \right)^{m_j}
\end{equation}
The latter, in case $f (x) = \log (x)$, can take the simpler form
\begin{equation}
  \label{formfaalog} (\log g (x))^{(n)} = \sum_{\{ \tmmathbf{m} \in
  \mathbb{N}^n \mid \sum_{j = 1}^n jm_j = n\}} C_{\tmmathbf{m}}  \prod_{j =
  1}^n \left( \frac{g^{(j)} (x)}{g (x)} \right)^{m_j}
\end{equation}
where
\begin{align*} C_{\tmmathbf{m}} = (- 1)^{1 + \sum_{j = 1}^n m_j} \frac{n! \left( \sum_{j =
   1}^n m_j - 1 \right) !}{m_1 !m_2 !...m_n !}  \prod_{j = 1}^n
   \frac{1}{j!^{m_j}} \end{align*}
and for $f (x) = e^x$,
\begin{equation}
  (e^{g (x)})^{(n)} = e^{g (x)} \sum_{\{ \tmmathbf{m} \in \mathbb{N}^n \mid
  \sum_{j = 1}^n jm_j = n\}} \frac{n!}{m_1 !m_2 !...m_n !}  \prod_{j = 1}^n
  \left( \frac{g^{(j)} (x)}{j!} \right)^{m_j} \label{formfaaexp}
\end{equation}
\begin{remark}
  \label{remarkbell}Fa{\`a} di Bruno formula \eqref{formfaa} is related to the polynomials that are known as Bell polynomials. The (complete) Bell polynomials are
  defined by the relation
  \begin{align} \mathcal{B}_{n} (x_1, x_2, \ldots, x_{n}) =
     \sum_{\{\tmmathbf{\kappa} \in \mathbbm{N}^{n}\mid \sum_{j=1}^nj\kappa_j=n\}} \frac{n!}{\kappa_1 !
     \kappa_2 ! \cdots \kappa_{n} !} \prod_{j = 1}^{n} \left(
     \frac{x_j}{j!} \right)^{\kappa_j}  \end{align}
  We can express the special case \eqref{formfaaexp} of Fa{\`a} di Bruno's formula for the exponential in terms of these Bell polynomials
  \begin{align*} \left(e^{g (x)}\right)^{(n)} = e^{g (x)} \mathcal{B}_n (g' (x), g'' (x), \ldots,
     g^{(n)} (x)) \end{align*}
\end{remark}
\begin{lemma}
  \label{lemmaA}Let $k, l \in \mathbbm{N}$. Then,  
  \begin{align}
    \lim_{a \rightarrow 0}  \frac{q (a) \log^k q (a)}{a^l} = 0
  \end{align}
\end{lemma}
\begin{proof}
  We have
  \begin{align*} \log \left( \frac{q (a)}{a^m} \right) = \log q (a) - m \log a \rightarrow
     - \infty \end{align*}
  for $a \rightarrow 0$, as, by \eqref{asymptotic0},
  \begin{align*} a (\log q (a) - m \log a) = a \log q (a) - ma \log a \rightarrow \log p
  \end{align*}
  This implies that
  \begin{align*} \lim_{a \rightarrow 0}  \frac{q (a)}{a^m} = 0 \end{align*}
  Also \eqref{asymptotic0} gives
  \begin{align*} \lim_{a \rightarrow 0} a^k \log^k q (a) = \log^k p \end{align*}
  Hence
  \begin{align*} \lim_{a \rightarrow 0}  \frac{q (a) \log^k q (a)}{a^l} = \lim_{a
     \rightarrow 0}  \frac{q (a)}{a^{l - k}}  \frac{\log^k q (a)}{a^k} = 0 \end{align*}
\end{proof}
\begin{lemma}
  \label{lemmaB}Let $N \in \mathbb{N}^{\ast}$ and assume that $\lim_{a
  \rightarrow 0} \varphi^{(k)} (a)$ exists in $\mathbb{R}$, $\forall k \leq
  N$. Then, $\forall k \leq N$,
  \begin{equation}
    \label{inlemmaB1} \lim_{a \rightarrow 0}  \frac{a^{2 k} q^{(k)} (a)}{q
    (a)}  \text{exists in } \mathbb{R}
  \end{equation}
  In particular, we have that
  \begin{equation}
    \label{inLemmaB2} \lim_{a \rightarrow 0}  \frac{q^{(k)} (a)}{a^m} = 0,
    \quad m \geq 0
  \end{equation}
\end{lemma}
\begin{proof}
  For $k = 1$, as $\varphi' (a) = - \log q (a) - aq' (a) / q (a)$, we have
  that  
  \begin{align*}
    \frac{a^2 q' (a)}{q (a)} = - a \varphi' (a) - a \log q (a) \rightarrow -
    \log p
  \end{align*}
  so \eqref{inlemmaB1} holds. Assume that $1 \leq n < N$ and that
  \eqref{inlemmaB1} holds $\forall k \leq n$. We will prove that
  \eqref{inlemmaB1} holds for $k = n + 1$. Indeed, using \eqref{formfaalog},
  we get, for some coefficients $c_{\tmmathbf{k}}$ and $d_{\tmmathbf{k}}$,
  \begin{align*}
     - \varphi^{(n + 1)} (a) =& a (\log q (a))^{(n + 1)} + (n + 1) (\log q
     (a))^{(n)} \\
     =& a \sum_{\{ \tmmathbf{k} \mid \sum_{j = 1}^{n + 1} jk_j = n + 1\}}
     \left[ c_{\tmmathbf{k}}  \prod_{j = 1}^{n + 1} \left( \frac{q^{(j)}
     (a)}{q (a)} \right)^{k_j} \right]\\ &+ (n + 1)  \sum_{\{ \tmmathbf{k} \mid
     \sum_{j = 1}^n jk_j = n\}} \left[ d_{\tmmathbf{k}}  \prod_{j = 1}^n
     \left( \frac{q^{(j)} (a)}{q (a)} \right)^{k_j} \right] \end{align*}  
  But one can write  
  \begin{align*}
    \sum_{\{ \tmmathbf{k} \mid \sum_{j = 1}^{n + 1} jk_j = n + 1\}} 
    c_{\tmmathbf{k}}  \prod_{j = 1}^{n + 1} \left( \frac{q^{(j)} (a)}{q (a)}
    \right)^{k_j} = \frac{q^{(n + 1)} (a)}{q (a)} + \sum_{\{
    \tmmathbf{k} \mid \sum_{j = 1}^n jk_j = n + 1\}} c_{\tmmathbf{k}} 
    \prod_{j = 1}^n \left( \frac{a^{2 j} q^{(j)} (a)}{q (a)} \right)^{k_j}
  \end{align*}  
  hence, rearranging the equation above and multiplying each side by $a^{2 n +
  1}$, we get
  \begin{align*}
     a^{2 (n + 1)}  \frac{q^{(n + 1)} (a)}{q (a)} =& - a^{2 n + 1}
     \varphi^{(n + 1)} (a) - \sum_{\{ \tmmathbf{k} \mid \sum_{j = 1}^n jk_j =
     n + 1\}} c_{\tmmathbf{k}}  \prod_{j = 1}^n \left( \frac{a^{2 j} q^{(j)}
     (a)}{q (a)} \right)^{k_j} \\
     &- a (n + 1)  \sum_{\{ \tmmathbf{k} \mid \sum_{j = 1}^n jk_j = n\}}
     d_{\tmmathbf{k}}  \prod_{j = 1}^n \left( \frac{a^{2 j} q^{(j)} (a)}{q
     (a)} \right)^{k_j} \end{align*}  
  and the right hand side converges in $\mathbb{R}$ as $a \rightarrow 0$ by
  our induction hypothesis, proving \eqref{inlemmaB1}. To prove
  \eqref{inLemmaB2}, we see that combining this result with Lemma \ref{lemmaA}
  gives
  \begin{align*} \lim_{a \rightarrow 0}  \frac{q^{(k)} (a)}{a^m} = \lim_{a \rightarrow 0} 
     \frac{a^{2 k} q^{(k)} (a)}{q (a)}  \frac{q (a)}{a^{m - 2 k}} = 0 \end{align*}
\end{proof}
\begin{lemma}
  \label{lemmaC}Let $N \in \mathbb{N}^{\ast}$ and assume that $\lim_{a
  \rightarrow 0} \varphi^{(k)} (a)$ exists in $\mathbb{R}$, $\forall k \leq
  N$. Then, $\forall k \leq N$,
  \begin{align*} \lim_{a \rightarrow 0} \rho^{(k)} (a) = 0, \quad k \neq 0 \end{align*}
  \begin{align*} \lim_{a \rightarrow 0} \rho (a) = 1 \end{align*}
\end{lemma}
\begin{proof}
  As $q (a) \rightarrow 0$, then $\rho (a) \rightarrow 1$. The $n$th
  derivative of $\rho$ can be expressed using \eqref{formfaa} as
  \begin{align*}
    \rho^{(n)} (a) = \sum_{\{ \tmmathbf{k} \mid \sum_{j = 1}^n jk_j = n\}}
    c_{\tmmathbf{k}}  (1 - q (a))^{b - 1 - \sum_{j = 1}^n k_j}  \prod_{j =
    1}^n (q^{(j)} (a))^{k_j}
  \end{align*}  
  which, by Lemma \ref{lemmaB} tends to $0$ as $a \rightarrow 0$, as $q^{(j)}
  (a) \rightarrow 0$.
\end{proof}
\begin{lemma}
  \label{lemmaD}Let $N \in \mathbb{N}^{\ast}$ and assume that $\lim_{a
  \rightarrow 0} \varphi^{(k)} (a)$ exists in $\mathbb{R}$, $\forall k \leq
  N$. Then, $\forall k \leq N$,
  \begin{align*} \lim_{a \rightarrow 0} \sigma^{(k)} (a) = 0 \end{align*}
\end{lemma}
\begin{proof}
  We have
  \begin{align*} \int_0^{q (a)} t^a  (1 - t)^{b - 2} \log^m t \mathd t \rightarrow 0 \end{align*}
  as $q (a) \rightarrow 0$ and $(1 - t)^{b - 2} \log^m t$ is integrable near
  $0$. Hence, $\sigma (a) \rightarrow 0$. For $n > 0$ we have
  \begin{align}
    \sigma^{(n)} (a) = \int_0^{q (a)} t^a  (1 - t)^{b - 2} \log^{n + 1} t
    \mathd t + \sum_{k = 1}^n [e^{- \varphi (a)} (1 - q (a))^{b - 2} q' (a)
    \log^k q (a)]^{(n - k)}
  \end{align}  
  So, it suffices to prove that
  \begin{align*} [e^{- \varphi (a)} (1 - q (a))^{b - 2} q' (a) \log^k q (a)]^{(l)}
     \rightarrow 0, \quad \forall k, l \leq N \end{align*}
  By \eqref{formprod} we can write
  \begin{align*}
	&[e^{- \varphi (a)}  (1 - q (a))^{b - 2} q' (a) \log^k q (a)]^{(l)} =\\
     &\sum_{\{ \tmmathbf{m} \mid \sum_{j = 1}^3 m_j = l\}} c_{\tmmathbf{m}}
    [e^{- \varphi (a)}]^{(m_1)} [(1 - q (a))^{b - 2}]^{(m_2)}  [q' (a) \log^k
    q (a)]^{(m_3)}
  \end{align*}  
  By our assumptions, $\lim_{a \rightarrow 0} [e^{- \varphi (a)}]^{(m_1)} \in
  \mathbb{R}$, and as in Lemma \ref{lemmaC}, $((1 - q (a))^{b - 2})^{(m_2)}$
  also converges. Finally, by \eqref{formprod}, \eqref{formfaa} and Lemma
  \ref{lemmaB}
  \begin{align*}
     [q' (a) \log^k q (a)]^{(m)} = \sum_{\{ \tmmathbf{n} \mid \sum_{j = 1}^{k
     + 1} n_j = m\}} c_{\tmmathbf{n}} q^{(n_1 + 1)} (a)  \prod_{j = 2}^{k + 1}
     [\log q (a)]^{(n_j)} = \\
     \sum_{\{ \tmmathbf{n} \mid \sum_{j = 1}^{k + 1} n_j = m\}}
     c_{\tmmathbf{n}} q^{(n_1 + 1)} (a)  \prod_{j = 2}^{k + 1} \sum_{\{r \mid
     \sum_{s = 1}^{n_j} sr_s = n_j \}} d_{\tmmathbf{r}}  \prod_{s = 1}^{n_j}
     \left( \frac{q^{(s)} (a)}{q (a)} \right)^{r_s} = \\
     \sum_{\{ \tmmathbf{n} \mid \sum_{j = 1}^{k + 1} n_j = m\}}
     c_{\tmmathbf{n}}  \frac{q^{(n_1 + 1)} (a)}{a^{2^k  \prod_{j = 2}^{k + 1}
     n_j}}  \prod_{j = 2}^{k + 1} \sum_{\{r \mid \sum_{s = 1}^{n_j} sr_s = n_j
     \}} d_{\tmmathbf{r}}  \prod_{s = 1}^{n_j} \left( \frac{q^{(s)} (a) a^{2
     s}}{q (a)} \right)^{r_s} \rightarrow 0 \end{align*}
  which completes the proof of the Lemma.
\end{proof}

\begin{proof*}{Proof of theorem \ref{theorem0}}
  By Proposition \ref{prop2} we have that $\varphi (0) = - \log p$. For the
  first derivative, as $\rho (0) = 1$ and $\sigma (0) = 0$, and $\varphi (0) =
  \psi (0) = - \log p$, we get from \eqref{eq1} that the limit $\lim_{a
  \rightarrow 0} \varphi' (a) = \varphi' (0)$ exists and $\varphi' (0) = \psi'
  (0)$. We proceed inductively. Let $n \in \mathbb{N}^{\ast}$ and assume that
  $\lim_{a \rightarrow 0} \varphi^{(k)} (a)$ exists and $\varphi^{(k)} (0) =
  \psi^{(k)} (0)$ $\forall k \leq n$. Differentiating \eqref{eq1} $n$ times we
  get  
  \begin{align*}
    (e^{- \psi (a)})^{(n + 1)} - (e^{- \varphi (a)})^{(n + 1)} \rho (a) -
    \sum_{k = 0}^{n - 1} (e^{- \varphi (a)})^{(k + 1)} \rho (a)^{(n - k)} = (b
    - 1) \sigma^{(n)} (a)
  \end{align*}  
  and by Lemmas \ref{lemmaC} and \ref{lemmaD} we get  
  \begin{align*}
    \lim_{a \rightarrow 0} (e^{- \psi (a)})^{(n + 1)} = \lim_{a \rightarrow 0}
    (e^{- \varphi (a)})^{(n + 1)}
  \end{align*}  
  which, by formula \eqref{formfaa} and the induction hypothesis, gives that
  the limit $\lim_{a \rightarrow 0} \varphi^{(n + 1)} (a) = : \varphi^{(n +
  1)} (0)$ exists in $\mathbb{R}$ and
  \begin{align*} \sum_{\{ \tmmathbf{k} \mid \sum_{j = 1}^{n + 1} jm_j = n + 1\}}
     c_{_{\tmmathbf{k}}} e^{- \psi (0)} \prod_{j = 1}^{n + 1}& \left(
     \frac{\psi^{(j)} (0)}{j!} \right)^{m_j} \\&=\sum_{\{ \tmmathbf{k} \mid
     \sum_{j = 1}^{n + 1} jm_j = n + 1\}} c_{\tmmathbf{k}} e^{- \varphi (0)}
     \prod_{j = 1}^{n + 1} \left( \frac{\varphi^{(j)} (0)}{j!} \right)^{m_j}
  \end{align*}  
  and as by the induction hypothesis $\varphi^{(j)} (0) = \psi^{(j)} (0)$ for
  $j \leq n$, it gives
  \begin{align*} \varphi^{(n + 1)} (0) = \psi^{(n + 1)} (0) \end{align*}
  which completes the induction. To prove \eqref{bintegerformula}, the fact
  that
  \begin{align*} \log \Gamma (x + 1) - \log \Gamma (x) = \log x \end{align*}
  gives the functional relation for the polygamma function  
  \begin{align}
    \Psi (k, x + 1) - \Psi (k, x) = \frac{(- 1)^k k!}{x^{k + 1}}
  \end{align}  
  hence
  \begin{align*} \varphi^{(k + 1)} (0) = \Psi (k, b) - \Psi (k, 1) = \sum_{n = 1}^{b - 1}
     (\Psi (k, n + 1) - \Psi (k, n)) = \sum_{n = 1}^{b - 1} \frac{(- 1)^k
     k!}{n^{k + 1}} \end{align*}
\end{proof*}

\begin{proof*}{Proof of Corollary \ref{cor0}}
  The fact that $\varphi$ and $\psi$ have the same asymptotic expansion at $0$
  implies that an approximation of $\varphi$ is
  \begin{align*} \varphi (a) \sim \log \frac{\Gamma (a + b)}{\Gamma (a + 1) \Gamma (b)} -
     \log p \quad \tmop{as} a \rightarrow 0 \end{align*}
  and the error decreases faster than any positive power of $a$. This also
  implies that
  \begin{align*} q (a) \sim \left( \frac{\Gamma (a + 1) \Gamma (b)}{\Gamma (a + b)}
     \right)^{1 / a} p^{1 / a} \quad \tmop{as} a \rightarrow 0 \end{align*}
  in the sense that $\forall n \in \mathbbm{N}, \varepsilon > 0, \exists a_{n,
  \varepsilon} > 0$ such that $\forall a < a_{n, \varepsilon}$
  \begin{align*} e^{- \varepsilon a^n} \left( \frac{\Gamma (a + 1) \Gamma (b)}{\Gamma (a +
     b)} \right)^{1 / a} p^{1 / a} < q (a) < e^{\varepsilon a^n} \left(
     \frac{\Gamma (a + 1) \Gamma (b)}{\Gamma (a + b)} \right)^{1 / a} p^{1 /
     a} \end{align*}
  hence
  \begin{equation}
    \lim_{a \rightarrow 0} \frac{q (a)}{p^{1 / a}} = e^{- \gamma - \Psi (0,
    b)}
  \end{equation}
  $\gamma$ being the Euler's constant. The RHS of the above inequality may be
  rewritten as
  \begin{align*} \frac{q (a)}{p^{1 / a}} < \left( \frac{\Gamma (a + 1) \Gamma (b)}{\Gamma
     (a + b)} \right)^{1 / a} + \varepsilon' a^n \left( \frac{\Gamma (a + 1)
     \Gamma (b)}{\Gamma (a + b)} \right)^{1 / a} \end{align*}
  close to $0$ and for an $\varepsilon' > \varepsilon$, and the LHS
  \begin{align*} \left( \frac{\Gamma (a + 1) \Gamma (b)}{\Gamma (a + b)} \right)^{1 / a} -
     \varepsilon a^n \left( \frac{\Gamma (a + 1) \Gamma (b)}{\Gamma (a +
     b)} \right)^{1 / a} < \frac{q (a)}{p^{1 / a}} \end{align*}
  Hence
  \begin{align*} \frac{q (a)}{p^{1 / a}} \sim \left( \frac{\Gamma (a + 1) \Gamma
     (b)}{\Gamma (a + b)} \right)^{1 / a} \end{align*}
  with a remainder term vanishing faster than any power of $a$ at $0$. The
  rest comes from considering
  \begin{align*} \left( \frac{\Gamma (a + 1) \Gamma (b)}{\Gamma (a + b)} \right)^{1 / a} =
     \exp \left( \frac{1}{a} \log \frac{\Gamma (a + 1) \Gamma (b)}{\Gamma (a +
     b)} \right) \end{align*}
  along with Fa{\`a} di Bruno formula.
\end{proof*}

\section{Asymptotics at $\infty$}

\subsection{The operator $D$}\label{sectionD}

To find the asymptotic expansion at infinity, the previous technique has to be
adjusted accordingly. First, we introduce the differential operator $D$
defined by
\begin{equation}
  D f (a) = a^2 \partial f (a) 
\end{equation}
It satisfies the product rule
\begin{equation}
  D (f g) (a) = g (a) D f (a) + f (a) D g (a)
\end{equation}
and the composition rule
\begin{align*} D (f \circ g) (a) = f' (g (a)) D g (a) \end{align*}
The last two relations combined give us the Faa di Bruno formula for $D$
\begin{equation}
  D^n (f \circ g) (a) = \sum_{\{ \tmmathbf{m} \in \mathbbm{N}^n \mid
  \sum_{j = 1}^n j m_j = n \}} \frac{n!}{m_1 !m_2 ! \ldots m_n !} f^{(|
  \tmmathbf{m} |)} (g (a)) \prod_{j = 1}^n \left( \frac{D^j g (a)}{j!}
  \right)^{m_j} \label{faaD}
\end{equation}
where $| \tmmathbf{m} | = \sum_{j = 1}^n m_j$. Also, we have the two-arguments composition rule
\begin{equation}
  D f (a, \varphi (a)) = D_1 f (a, \varphi (a)) + D \varphi (a) \partial_2 f
  (a, \varphi (a))
\end{equation}
where $D_1 f (a, b) = a^2 (\partial_1 f) (a, b)$, $\partial_1$ denoting
differentiation wrt the first variable of a multivariate function, i.e. in our
case $D_1 f (a, \varphi (a)) = a^2 (\partial_1 f) (a, \varphi (a))$. Furthermore, we
remark that it acts on monomials, for $m \in \mathbbm{Z}$, by
\begin{align*} D a^m = n a^{m + 1} \end{align*}
and by induction
\begin{align*} D^n a^m = (m)_n a^{m + n} \end{align*}
The operator $D$ can be used to deal with asymptotic expansions at infinity.
To see this, intuitively, starting from the formal power series
\begin{align*} f (x) = c_0 + \frac{c_1}{x} + \frac{c_2}{x^2} + \frac{c_3}{x^3} + \ldots \end{align*}
one can get
\begin{align*} D^n f (x) = \sum^{\infty}_{k = n} (- 1)^n \frac{k!}{(n - k) !}
   \frac{c_k}{x^{k - n}} \end{align*}
If certain conditions apply and it is possible to take limits to $\infty$, all
but the first term of the sum vanish and we get
\begin{align*} \lim_{x \rightarrow \infty} D^n f (x) = (- 1)^n n!c_n \end{align*}
This is rigorously treated in Lemma \ref{lemD}, which is proved below:

\begin{proof*}{Proof of Lemma \ref{lemD}}
  To show i), we notice that if for a function $f$ we have $\lim_{x
  \rightarrow \infty} f (x) = a_0 \in \mathbbm{R}$ and $D f (x) = a_1 + a_2 /
  x + a_3 / x^2 + \ldots + a_{k + 1} / x^k +\mathcal{O} (1 / x^{k + 1})$, then
  by integrating we get that $f (x) = a_0 - a_1 / x - a_2 / 2 x^2 - a_3 / 3
  x^3 + \ldots + a_{k + 1} / k x^{k + 1} +\mathcal{O} (1 / x^{k + 2})$. Next,
  we see that, under the assumptions of the first part of the lemma, we have
  that $\lim_{x \rightarrow \infty} D^{n - 1} f (x) = a \in \mathbbm{R}$ and
  $\lim_{x \rightarrow \infty} D^n f (x) = b \in \mathbbm{R}$. This implies
  that $D^{n - 1} f (x) = a +\mathcal{O} (1 / x)$. Applying this observation
  inductively to find the asymptotic expansions of lower powers of $D$ proves
  the first part of the Lemma. For the second part, we notice that as the
  derivatives admit asymptotic expansions, these can be obtained by
  differentiating the asymptotic expansion of the original function. In the
  same way, we may apply the operator $D$ to the original asymptotic
  expansion, as $D^k$ can be expressed as a combination of operators
  $\partial^l$ for $l \leqslant k$, and take limits to $\infty$ to prove the
  second part.\end{proof*}\\In the following subsections we shall compute the asymptotic expansion of
$\varphi$ using the operator $D$. We start with the equation
\begin{equation}
  \int_0^{\varphi (a)} \tau (a ; s) \mathd s = (1-p) \frac{\Gamma (b) \Gamma (a)
  a^b}{\Gamma (a + b)} \label{eq0}
\end{equation}
Where
\begin{equation}
  \tau (a ; s) = e^{- s} (a - a e^{- s / a})^{b - 1} \label{tau}
\end{equation}
Our method consists of acting and iterating the operator $D$ on \eqref{eq0}
and taking the limits to $\infty$ on both sides. So we have to see how $D$
acts on $\tau$ and on the right hand side.

\subsection{Asymptotics of the RHS}

To study the right hand side of the equation \eqref{eq0}, we study the
asymptotics of the ratio
\begin{equation}
  \frac{\Gamma (a) a^b}{\Gamma (a + b)} \label{ratio}
\end{equation}
In {\cite{tricomi}}, Tricomi and Erdelyi derived an asymptotic expansion for
such ratios of Gamma functions, in terms of a generalisation of N{\o}rlund
Polynomials, which in our special case it may be expressed as
\begin{align*} \frac{\Gamma (a) a^b}{\Gamma (a + b)} \sim \sum_{n \geqslant 0}
   \frac{\Gamma (1 - b)}{\Gamma (1 - (b + n))} \frac{B_n^{(1 - b)}}{n!a^n},
   \quad x \rightarrow \infty \end{align*}
which by the reflection formula for the Gamma function can be rewritten as
\begin{equation}
  \frac{\Gamma (a) a^b}{\Gamma (a + b)} \sim \sum_{n = 0}^{\infty} \frac{(-
  1)^n}{n!} (b)_n \frac{B_n^{(1 - b)}}{a^n} \label{asratio}
\end{equation}
We shall prove the following Lemma:
\begin{lemma}
  For $n \in \mathbbm{N}$, we have\label{lemmarhs}
  \begin{equation}
    \lim_{a \rightarrow \infty} D^n \left( (1-p) \frac{\Gamma (b) \Gamma (a)
    a^b}{\Gamma (a + b)} \right) = (1-p) \Gamma (b + n) B_n^{(1 - b)}
    \label{limRHS}
  \end{equation}
\end{lemma}

\begin{proof}
  The coefficients of the asymptotic expansion \eqref{asratio}, by Lemma
  \ref{lemD}, can be used to give the limit in \eqref{limRHS}, if the
  derivatives of the ratio also admit asymptotic expansions. Hence we shall
  find these asymptotic expansions of the derivatives, and also a different
  expression for the coefficients in the asymptotic expansion of the ratio
  \eqref{ratio} on the way.
  
  The tool we shall work with is the operator $D$ and its Fa{\`a} di
  Bruno formula eq\ref{faaD}. We denote the logarithmic derivative of the
  ratio \eqref{ratio} by
  \begin{equation}
    V (a) \assign \log \frac{\Gamma (a) a^b}{\Gamma (a + b)} = b \log a + \log
    \Gamma (a) - \log \Gamma (a + b)
  \end{equation}
  A classic result on the asymptotic expansion of $\log \Gamma$ is the
  following, see {\cite[5.11.8]{DLMF}}, for fixed
  $h \in \mathbbm{C}$,
  \begin{equation}
    \log \Gamma (x + h) \sim \log \sqrt{2 \pi} + \left( x + h - \frac{1}{2}
    \right) \log x - x + \sum_{n \geq 2} \frac{B_n (h)}{n (n - 1)} x^{1 - n}
    \label{asloggamma}, \quad x \rightarrow + \infty
  \end{equation}
  which has the nice property that it can also be differentiated, and give us
  asymptotic expansions of polygamma functions. This implies also that the
  derivatives of $V$ admit asymptotic expansions. We have, asymptotically,
  \begin{align*} \begin{array}{l}
       V (a) \sim \sum_{n \geq 2} \frac{B_n - B_n (b)}{n (n - 1)}
       \frac{1}{a^{n - 1}} = \sum_{n \geq 1} \frac{B_{n + 1} - B_{n + 1}
       (b)}{n (n + 1)} \frac{1}{a^n}
     \end{array} \end{align*}
  We have, then, by Lemma \ref{lemD}, that
  \begin{equation}
    \lim_{a \rightarrow \infty} D^n V (a) = (- 1)^n n! \frac{B_{n + 1} - B_{n
    + 1} (b)}{n (n + 1)}
  \end{equation}
  Acting $D$ $n$ times on ratio \eqref{ratio} we get
  \begin{align}
    &D^n \left(  \frac{\Gamma (a) a^b}{\Gamma (a + b)} \right) = D^n e^{V (a)}\nonumber
    \\
    &\qquad= e^{V (a)} \sum_{\{ \tmmathbf{m} \in \mathbbm{N}^n \mid \sum_{j =
    1}^n j m_j = n \}} \frac{n!}{m_1 !m_2 ! \ldots m_n !} \prod_{j =
    1}^n \left( \frac{D^j V (a)}{j!} \right)^{m_j} \label{rhs0}
  \end{align}
  and taking limits we end up with
  \begin{align*} \lim_{a \rightarrow \infty} D^n \left(  \frac{\Gamma (a) a^b}{\Gamma (a +
     b)} \right) = \sum_{\{ \tmmathbf{m} \in \mathbbm{N}^n \mid \sum_{j =
     1}^n j m_j = n \}} \frac{(- 1)^n n!}{m_1 !m_2 ! \ldots m_n !}
     \prod_{j = 1}^n \left( \frac{B_{j + 1} - B_{j + 1} (b)}{j (j + 1)}
     \right)^{m_j} \end{align*}
  Hence, by Lemma \ref{lemD}, the derivatives of the ratio \eqref{ratio} admit
  asymptotic expansions at infinity, and these can be given by differentiating
  the asymptotic expansion \eqref{asratio}. 
\end{proof}

\begin{remark}
  \label{remarkratio}In the proceeding proof, we find two different ways to
  express the asymptotic expansion of the ratio of gamma functions, which
  implies a relation between N{\o}rlund, Bernoulli and Bell polynomials we
  could not trace in the literature,
  \begin{equation}
    (b)_n B_n^{(1 - b)} = \mathcal{B}_n (B_2 (b) - B_2, B_3 (b) - B_3, \ldots,
    B_{n + 1} (b) - B_{n + 1} )
  \end{equation}
  and using the fact that
  \begin{align*} \sum^{j + 1}_{k = 1} B_{j - k + 1} \frac{(j - 1) !}{k! (j + 1 - k) !} b^k
     = B_{j + 1} (b) - B_{j + 1} \end{align*}
  we get

  \begin{align*} (b)_n B_n^{(1 - b)} = \sum_{\{ \tmmathbf{m} \in \mathbbm{N}^n \mid
     \sum_{j = 1}^n j m_j = n \}} \frac{n!}{m_1 !m_2 ! \ldots m_n !}
     \prod_{j = 1}^n \left( \sum^{j + 1}_{k = 1} B_{j - k + 1} \frac{(j - 1)
     !}{k! (j - k + 1) !} b^k \right)^{m_j} \end{align*}
\end{remark}

\subsection{Asymptotics of the LHS}

We shall first study the asymptotic behaviour of $\tau$, defined in
\eqref{tau}, through the following Lemma.
\begin{lemma}
  We have the limits
  \begin{equation}
    \lim_{a \rightarrow \infty} D^n \tau (a ; s) = B_n^{(1 - b)} e^{- s} s^{b
    - 1 + n}
  \end{equation}
\end{lemma}
\begin{proof}
  We have that
  \begin{align*} \tau (a ; s) = e^{- s} (a - a e^{- s / a} )^{b - 1} = e^{- s} \left(
     \frac{\frac{1}{a}}{1 - e^{- s / a}} \right)^{1 - b} = e^{- s} s^{b - 1}
     \left( \frac{- \frac{s}{a}}{e^{- s / a} - 1} \right)^{1 - b} \end{align*}
  We may write, in terms of N{\o}rlund polynomials, by \eqref{norlund},
  \begin{equation}
    \tau (a ; s) = e^{- s} s^{b - 1} \sum_{k = 0}^{\infty} B_k^{(1 - b)}
    \frac{(- 1)^k s^k}{k!a^k}
  \end{equation}
  and
  \begin{align*} D^n \tau (a ; s) = e^{- s} s^{b - 1} \sum_{k = n}^{\infty} B_k^{(1 - b)}
     \frac{(- 1)^{k + n} s^k}{(k - n) !a^{k - n}} \end{align*}
  and thus we get
  \begin{align*} \lim_{a \rightarrow \infty} D^n \tau (a ; s) = B_n^{(1 - b)} e^{- s} s^{b
     - 1 + n} 
    \end{align*}
\end{proof}\\Acting $D$ on the left hand side of \eqref{eq0} gives the expression
\begin{align*} D \int_0^{\varphi (a)} \tau (a ; s) \mathd s = D \varphi (a) \tau (a ;
   s) + \int_0^{\varphi (a)} D \tau (a ; s) \mathd s \end{align*}
and hence by induction, iterating $D$ totally $n$ times,
\begin{equation}
  D^n \int_0^{\varphi (a)} \tau (a ; s) \mathd s = \sum_{k = 0}^{n - 1} D^k
  (\nobracket D \varphi (a) D^{n - k - 1}_1 \tau (a ; \varphi (a)
  \nobracket)) + \int_0^{\varphi (a)} D^n \tau (a ; s) \mathd s
\end{equation}
We shall study the terms
\begin{align*} D^k (\nobracket D \varphi (a) D^{n - k - 1}_1 \tau (a ; \varphi (a)
   \nobracket)) = \sum_{j = 0}^k \binom{k}{j} D^{k - j + 1} \varphi (a) D^j
   [D_1^{n - k - 1} \tau (a ; \varphi (a))] \end{align*}
and as
\begin{align*} D^j [D_1^{n - k - 1} \tau (a ; \varphi (a))] = D^{j - 1} [D_1^{n - k} \tau
   (a ; \varphi (a)) + D \varphi (a) D_1^{n - k - 1} \partial_2 \tau (a ;
   \varphi (a))] \end{align*}
it is important to study the terms defined as
\begin{equation}
  d (k, m, n) \assign \lim_{a \rightarrow \infty} D^m [D_1^n \partial_2^k \tau
  (a ; \varphi (a))]
\end{equation}
In other words, we will compute, recursively, the limits of these terms for $a
\rightarrow \infty$. We note that, as $D^n \tau (a ; s)$ is an analytic function
of $s$ in some disc around $0$, as seen by its power series, we can
interchange differentiation wrt the second variable and the limit for $a
\rightarrow \infty$, as we know that $\varphi (a)$ converges to a finite
limit, provided that the convergence for $a \rightarrow \infty$ is locally
uniform, which indeed is (an argument: as $a \rightarrow \infty$, the radius
of convergence of the power series increase, so taking a compact set and
assuming $a$ large enough, we can use the convergence of the sequence of power
series to prove this result). We have
\begin{align*} D^m [D_1^n \partial_2^k \tau (a ; \varphi (a))] = D^{m - 1} [D_1^{n + 1}
   \partial_2^k \tau (a ; \varphi (a)) + D \varphi (a) D_1^n \partial_2^{k + 1}
   \tau (a ; \varphi (a))] \end{align*}
hence we get the recursive relation
\begin{equation}
  d (k, m, n) = d (k, m - 1, n + 1) + \sum_{j = 0}^{m - 1} \binom{m - 1}{j}
  \varphi_{m - j} d (k + 1, j, n) \label{recursd}
\end{equation}
where $\varphi_l = \lim_{a \rightarrow \infty} D^l \varphi (a)$, assuming that
the limit is already known, and the boundary conditions
\begin{align*} \\
   d (k, 0, n) = \lim_{a \rightarrow \infty} D_1^n \partial_2^k \tau (a ;
   \varphi (a)) = \lim_{a \rightarrow \infty} \partial_2^k D_1^n \tau (a ;
   \varphi (a)) \\
   = B_n^{(1 - b)} \sum_{j = 0}^k \binom{k}{j} (- 1)^{k - j} (b + n - j)_j
   e^{- \gamma_b} \gamma_b^{b - 1 + n - j} \end{align*}
As for the integral term, we have
\begin{equation}
  \lim_{a \rightarrow \infty} \int_0^{\varphi (a)} D^n \tau (a ; s) \mathd s =
  B_n^{(1 - b)} \int_0^{\gamma_b} e^{- s} s^{b - 1 + n} \mathd s
\end{equation}
and
\begin{equation}
  \int_0^{\gamma_b} e^{- s} s^{b - 1 + n} \mathd s = - \sum_{k = 0}^{n - 1} (b
  + n - k)_k e^{- \gamma_b} \gamma_b^{b - 1 + n - k} + (b)_n (1-p) \Gamma (b)
\end{equation}
by repeated integrations by parts and the fact that $\int_0^{\gamma_b} e^{- s}
s^{b - 1} \mathd s = (1-p) \Gamma (b)$. We have got, then, for the left hand side
that
\begin{align*} &\lim_{a \rightarrow \infty} D^n \int_0^{\varphi (a)} \tau (a ; s) \mathd s \\
   &\quad=\lim_{a \rightarrow \infty} \sum_{k = 0}^{n - 1} D^k (\nobracket D \varphi
   (a) D^{n - k - 1}_1 f (a ; \varphi (a) \nobracket)) + \lim_{a
   \rightarrow \infty} \int_0^{\varphi (a)} D^n f (a ; s) \mathd s \\&\quad=
   \lim_{a \rightarrow \infty} \sum_{k = 0}^{n - 1} \sum_{j = 0}^k
   \binom{k}{j} D^{k - j + 1} \varphi (a) D^j [D_1^{n - k - 1} f (a ; \varphi
   (a))] \\&
   \qquad- B_n^{(1 - b)} \sum_{k = 0}^{n - 1} (b + n - k)_k e^{- \gamma_b}
   \gamma_b^{b - 1 + n - k} + (b)_n (1-p) \Gamma (b) B_n^{(1 - b)} \\&\quad
   = \sum_{k = 0}^{n - 1} \sum_{j = 0}^k \binom{k}{j} \varphi_{k - j + 1} d
   (0, j, n - k - 1) \\&\qquad- B_n^{(1 - b)} \sum_{k = 0}^{n - 1} (b + n - k)_k e^{-
   \gamma_b} \gamma_b^{b - 1 + n - k} + (1-p) \Gamma (b + n) B_n^{(1 - b)} \\&\quad
   = \varphi_n d (0, 0, 0) + \sum_{j = 1}^{n - 1} \binom{n - 1}{j} \varphi_{n
   - j} d (0, j, 0) \\&\qquad
   + \sum_{k = 0}^{n - 2} \sum_{j = 0}^k \binom{k}{j} \varphi_{k - j + 1} d
   (0, j, n - k - 1)\\&\qquad - B_n^{(1 - b)} \sum_{k = 0}^{n - 1} (b + n - k)_k e^{-
   \gamma_b} \gamma_b^{b - 1 + n - k} + (1-p) \Gamma (b + n) B_n^{(1 - b)} \\
   \end{align*}
We notice that the term $(1-p) \Gamma (b + n) B_n^{(1 - b)}$ cancels exactly with
the right hand side.
\subsection{Conclusion}
\begin{proof*}{Proof of Theorem \ref{theoreminf}}
Summing up, using the normalisation $\delta = \frac{d}{e^{- \gamma_b}
\gamma_b^{b - 1}}$, we are left with
\begin{align*}  \\
   \varphi_n = - \sum_{j = 1}^{n - 1} \binom{n - 1}{j} \varphi_{n - j} \delta
   (0, j, 0) - \sum_{k = 0}^{n - 2} \sum_{j = 0}^k \binom{k}{j} \varphi_{k - j
   + 1} \delta (0, j, n - k - 1) \\
   + B_n^{(1 - b)} \sum_{k = 0}^{n - 1} (b + n - k)_k \gamma_b^{n - k} \end{align*}
and $\delta$ and $\varphi$ also satisfying the recursive relation, by
\eqref{recursd},
\begin{align*} \delta (k, m, n) = \delta (k, m - 1, n + 1) + \sum_{j = 0}^{m - 1} \binom{m
   - 1}{j} \varphi_{m - j} \delta (k + 1, j, n) \end{align*}
We have the initial conditions
\begin{align*} \varphi_0 = \gamma_b \end{align*}
\begin{align*} \delta (k, 0, n) = B_n^{(1 - b)} \sum_{j = 0}^k \binom{k}{j} (- 1)^{k - j}
   (b + n - j)_j \gamma_b^{n - j} \end{align*}\end{proof*} \\
To prove Corollary \ref{corinf} we need the following lemma.
\begin{lemma}\label{lemasymexpbell}
	Let $f$ have asymptotic expansion $$f(x)=\sum_{k=0}^N\frac{a_k}{k!x^k}+r(x)$$ where $r(x)=\mathcal{O}(1/x^{N+1})$. Then, \[e^{f(x)}=e^{a_0}+e^{a_0}\sum_{k=1}^N\frac{\mathcal{B}_k(a_1,a_2,...,a_k)}{k!x^k}+\mathcal{O}(1/x^{N+1})\]
\end{lemma}
\begin{proof}
	We have
	\begin{align*}
	f(x)&=\sum_{k=0}^N\frac{a_k}{k!x^k}+r(x)\Rightarrow
	e^{f(x)-\sum_{k=0}^N\frac{a_k}{k!x^k}}=e^{r(x)}=1+\mathcal{O}(1/x^{N+1})\\
	\Rightarrow e^{f(x)}&=e^{\sum_{k=0}^N\frac{a_k}{k!x^k}}+\mathcal{O}(1/x^{N+1})=e^{a_0}\prod_{k=1}^Ne^{\frac{a_k}{k!x^k}}+\mathcal{O}(1/x^{N+1})\\
	&=e^{a_0}\prod_{k=1}^N
	 \left(1+\sum_{m=1}^{\left\lceil\frac{N+1}{k}-1\right\rceil}\frac{a_k^m}{m!k!^mx^{km}}+\mathcal{O}(1/x^{N+1})\right)+\mathcal{O}(1/x^{N+1})\\
	&=e^{a_0}\sum_{n=0}^N\frac{\mathcal{B}_n(a_1,a_2,...,a_n)}{n!x^n}+\mathcal{O}(1/x^{N+1})
	\end{align*}
	where the last equality is derived by a combinatorial argument, the coefficient of $1/x^n$ being the sum of products of the form $\prod_{k=1}^n \frac{a_k^{m_k}}{(k!)^{m_k}m_k!}$ such that $\sum_{k=1}^{n}km_k=n$, which defines the complete Bell polynomials.
\end{proof}

\begin{proof*}{Proof of Corollary \ref{corinf}}
	The proof is an immediate consequence of Theorem \ref{theoreminf} and the foregoing Lemma.
\end{proof*}

\section{Relations between Bell, Bernoulli and N{\o}rlund Polynomials}
In the course of trying to find the asymptotic expansion of $\varphi$ at
$\infty$, using Fa\`a di Bruno formulas, we encountered identities between Bell
polynomials and N{\o}rlund polynomials, that we have not been able to trace in
the literature, hence we state them in this section as a separate result.
\begin{proposition}
  Let $c \in \mathbbm{C}$. Then, the N{\o}rlund polynomial $B_n^{(c)}$ can be
  expressed as
  \begin{equation}
    B^{(c)}_n = \sum_{\{ \tmmathbf{m} \in \mathbbm{N}^n \mid \sum_{j =
    1}^n j m_j = n \}} \frac{n!}{m_1 !m_2 ! \ldots m_n !} \prod_{j =
    1}^n \left( \frac{(- 1)^{j + 1} c B_j}{j!j} \right)^{m_j}
  \end{equation}
  or, phrased in terms of Bell polynomials $\mathcal{B}_n$,
  \begin{equation}
    B^{(c)}_n = \mathcal{B}_n (c B_1, - c B_2 / 2, 0, - c B_4 / 4, 0,
    \ldots, - c B_n / n), \quad n > 1
  \end{equation}
  Moreover, we have that
  \begin{equation}
    (c - n)_n B_n^{(c)} = (- 1)^n \mathcal{B}_n (B_2 (c) - B_2, - B_3 (c) -
    B_3, \ldots, (- 1)^{n + 1} B_{n + 1} (c) - B_{n + 1} )
  \end{equation}
\end{proposition}
\begin{proof}
  The last equation is derived by Remark \ref{remarkratio}, and the symmetries
  $B_n (1 - x) = (- 1)^n B_n (x)$ and $(1 - c)_n = (- 1)^n (c - n)_n$. For the
  rest, by \eqref{norlund} we have
  \begin{align*} B_n^{(c)} = \lim_{z \rightarrow 0} \partial^n \left[ \left( \frac{z}{e^z
     - 1} \right)^c \right] \end{align*}
  By using Fa{\`a} di Bruno formula we get
  \begin{align*}
     &\partial^n \left[ \left( \frac{z}{e^z - 1} \right)^c \right] =
     \partial^n \left( e^{c \log \frac{z}{e^z - 1}} \right) \\
     &\quad= \left( \frac{z}{e^z - 1} \right)^c \sum_{\{ \tmmathbf{m} \in
     \mathbbm{N}^n \mid \sum_{j = 1}^n j m_j = n \}} \frac{n!}{m_1 !m_2
     ! \ldots m_n !} \prod_{j = 1}^n \left( \frac{c}{j!} \partial^j \left(
     \log \frac{z}{e^z - 1} \right) \right)^{m_j} \end{align*}
  and we have the limit
  \begin{align*} \lim_{z \rightarrow 0} \left( \frac{z}{e^z - 1} \right)^c = 1 \end{align*}
  and
  \begin{align*} - z \left( \log \frac{e^z - 1}{z} \right)' = - \frac{z e^z}{e^z - 1} + 1
     &= \sum_{n = 1}^{\infty} (- 1)^{n + 1} B_n \frac{z^n}{n!} \\\Rightarrow \log
     \frac{z}{e^z - 1} &= \sum_{n = 1}^{\infty} (- 1)^{n + 1} \frac{B_n}{n}
     \frac{z^n}{n!} \end{align*}
  hence
  \begin{align*} \lim_{z \rightarrow 0} \partial^j \left( \log \frac{z}{e^z - 1} \right) =
     (- 1)^{j + 1} \frac{B_j}{j} \end{align*}
  and thus, summing up,
  \begin{align*} B_n^{(c)} &= \lim_{a \rightarrow 0} \partial^n \left[ \left( \frac{z}{e^z
     - 1} \right)^c \right] \\&= \sum_{\{ \tmmathbf{m} \in \mathbbm{N}^n
     \mid \sum_{j = 1}^n j m_j = n \}} \frac{n!}{m_1 !m_2 ! \ldots m_n
     !} \prod_{j = 1}^n \left( \frac{(- 1)^{j + 1} c B_j}{j!j} \right)^{m_j}
  \end{align*}
  which concludes the proposition.
\end{proof}

\appendix
\section*{Appendix}

In this appendix, we provide code in Maple and Sage for computing the terms
of asymptotic expansion of $\varphi$ and $q$ at infinity recursively.

\section{Maple code}

In the first algorithm, the procedure \texttt{phiinf(n)} computes what we
define as $\varphi_n$ in Theorem \ref{theoreminf}.
\begin{algorithm} \quad\\
\begin{lstlisting}
    norl:=proc(n,c);
		if n=0 then
	        return 1;
        else
            return (-1)^n*CompleteBellB(
    n,seq(-c*bernoulli(j)/j,j=1..n)
            );
        end if;
    end proc;
    phiinf:=proc(n) option remember;
	    if n=0 then
		    return gamma[b];
	    end if;
        for k from 1 to n-1 do
           phiinf(k);
        od;
        return expand(
    -add(binomial(n-1,j)*phiinf(n-j)*delta(0,j,0),j=1..n-1)
    -add(add(binomial(k,j)*phiinf(k-j+1)*delta(0,j,n-k-1)
    ,j=0..k),k=0..n-2)+norl(n,1-b)*add(pochhammer(b+n-k,k)
    *phiinf(0)^(n-k),k=0..n-1);
    end proc;
    delta:=proc(k,m,n) option remember;
	    if m=0 then
		    return simplify(
    norl(n,1-b)*add(binomial(k,j)*(-1)^(k-j)*
    pochhammer(b+n-j,j)*phiinf(0)^(n-j),j=0..k)
            );
	    end if;
        return simplify(delta(k,m-1,n+1)+
    add(binomial(m-1,j)*phiinf(m-j)*delta(k+1,j,n),j=0..m-1));
    end proc;
  \end{lstlisting}
\end{algorithm}
In the second algorithm, the procedure \texttt{qinf(n)} computes the nth
coefficient of the asymptotic expansion of $q$ in Corollary \ref{corinf}.
\begin{algorithm} \quad\\
\begin{lstlisting}
qinf:=proc(b,n);
   if n=0 then 
      return 1; 
   else 
      return CompleteBellB(
      n,seq((-1)^(k+1)*phiinf(k),k=0..n-1))/n!;
   end if;
end proc;
\end{lstlisting}
\end{algorithm}
\section{Sage code}

The function \texttt{phiinf(n)} computes what we define as $\varphi_n$ in
Theorem \ref{theoreminf}.
\begin{algorithm} \quad\\
\begin{lstlisting}
gamma_b=var('gamma_b')
    b=var('b')
    def norlund(n,c):
       if n==0:
           return 1
       else:
           return
    (-1)^n*sum(bell_polynomial(n,k)([-c*bernoulli(j)/j for j in
    [1..n-k+1]]) for k in [1..n])
    @CachedFunction
    def phiinf(n):
       if n==0:
           return gamma_b
       else:
           return expand(-sum(binomial(n-1,j)*phiinf(n-j)*
    delta(0,j,0) for j in [1..n-1])-sum(sum(binomial(k,j)*
    phiinf(k-j+1)*delta(0,j,n-k-1) for j in [0..k]) for k 
    in [0..n-2])+norlund(n,1-b)*sum(rising_factorial(b+n-k,k)*
    phiinf(0)^(n-k) for k in [0..n-1]))
    @CachedFunction
    def delta(k,m,n):
		if m==0:
           return
    simplify(norlund(n,1-b)*sum(binomial(k,j)*(-1)^(k-j)*
    rising_factorial(b+n-j,j)*phiinf(0)^(n-j) for j in [0..k]))
	    else:
           return
    simplify(delta(k,m-1,n+1)+sum(binomial(m-1,j)
    *phiinf(m-j)*delta(k+1,j,n)
    for j in [0..m-1]))
  \end{lstlisting}
\end{algorithm}
The function \texttt{qinf(n)} computes the nth coefficient of the asymptotic
expansion of $q$ in Corollary \ref{corinf}.
\begin{algorithm} \quad\\
\begin{lstlisting}
    def qinf(n):
	    return sum(bell_polynomial(n,j)(
	    [(-1)^(k+1)*phiinf(k)
    for k in [0..n-j]]) for j in [1..n])/factorial(n)
\end{lstlisting}
\end{algorithm}
\bibliographystyle{plain}
\bibliography{betabib}

\begin{thebibliography}{10}

\bibitem{gammaram3}
J.~A. {Adell} and P.~{Jodrá}.
\newblock {On a Ramanujan equation connected with the median of the gamma
  distribution}.
\newblock {\em {Trans. Amer. Math. Soc.}}, 360:3631--3644, 2008.

\bibitem{gammaram1}
S.~E. Alm.
\newblock Monotonicity of the difference between median and mean of gamma
  distributions and of a related ramanujan sequence.
\newblock {\em Bernoulli}, 9(2):351--371, 2003.

\bibitem{special_functions}
G.~E. Andrews, R.~Askey, and R.~Roy.
\newblock {\em Special Functions}.
\newblock Cambridge University Press, 1999.
\newblock Cambridge Books Online.

\bibitem{logconc}
D.~Askitis.
\newblock Logarithmic concavity of the inverse incomplete beta function with
  respect to parameters.
\newblock Preprint.

\bibitem{chenrubin}
C.~Berg and H.~L. Pedersen.
\newblock The chen-rubin conjecture in a continuous setting.
\newblock {\em Methods and Applications of Analysis}, 13, 2006.

\bibitem{mediangammaconvex}
C.~Berg and H.~L. Pedersen.
\newblock Convexity of the median in the gamma distribution.
\newblock {\em Ark. Mat.}, 46, 2008.

\bibitem{chenrubin0}
J.~{Chen} and H.~{Rubin}.
\newblock {Bounds for the difference between median and mean of gamma and
  Poisson distributions.}
\newblock {\em {Stat. Probab. Lett.}}, 4:281--283, 1986.

\bibitem{gammaram2}
K.~P. Choi.
\newblock On the medians of gamma distributions and an equation of ramanujan.
\newblock {\em Proceedings of the American Mathematical Society},
  121(1):245--251, 1994.

\bibitem{DLMF}
{\it NIST Digital Library of Mathematical Functions}.
\newblock http://dlmf.nist.gov/, Release 1.0.12 of 2016-09-09.
\newblock F.~W.~J. Olver, A.~B. {Olde Daalhuis}, D.~W. Lozier, B.~I. Schneider,
  R.~F. Boisvert, C.~W. Clark, B.~R. Miller and B.~V. Saunders, eds.

\bibitem{ratiogamma}
D.~B. Karp and E.~G. Prilepkina.
\newblock Completely monotonic gamma ratio and infinitely divisible h-function
  of fox.
\newblock {\em Computational Methods and Function Theory}, 16(1):135--153,
  2015.

\bibitem{beta1}
M.~E. Payton, L.~J. Young, and J.~H. Young.
\newblock Bounds for the difference between median and mean of beta and
  negative binomial distributions.
\newblock {\em Metrika}, 36(1):347--354, 1989.

\bibitem{temme}
N.M. Temme.
\newblock Asymptotic inversion of the incomplete beta function.
\newblock {\em Journal of Computational and Applied Mathematics}, 41(1):145 --
  157, 1992.

\bibitem{tricomi}
F.~G. Tricomi and A.~Erdélyi.
\newblock The asymptotic expansion of a ratio of gamma functions.
\newblock {\em Pacific J. Math.}, 1(1):133--142, 1951.

\end{thebibliography}

\end{document}